\documentclass[reqno,11pt]{amsart}
\usepackage{url}
\usepackage{amssymb,amsthm,amsmath,amsfonts,amstext,mathrsfs}
\usepackage{latexsym}
\usepackage[all]{xy}

\usepackage[T1]{fontenc}
\usepackage{lmodern,amsmath,amssymb,amsthm,microtype}
\usepackage[textwidth=16.4cm,textheight=23.6cm,centering]{geometry}
\usepackage[hidelinks]{hyperref}

\newtheorem{thm}{Theorem}
\newtheorem{lem}{Lemma}
\newtheorem{prop}{Proposition}

\def\d{\,\mathrm{d}}
\newcommand{\m}{\mathfrak m}
\newcommand{\W}{\mathcal W}
\newcommand{\E}{\mathbb E}
\newcommand{\Pp}{\mathbb P}
\newcommand{\norm}[1]{\lVert#1\rVert}
\newcommand{\ind}{\mathbf 1}
\title[Giedrius Alkauskas]
{The Minkowski $?(x)$ function and Salem's problem. II}

\author[Giedrius Alkauskas]{Giedrius Alkauskas}
\begin{document}

\begin{abstract}
In 1943, R. Salem asked whether the Fourier--Stieltjes transform of the Minkowski question-mark function vanishes at infinity. This problem was answered affirmatively by Jordan and Sahlsten in 2016 as a consequence of their general results on Gibbs measures for the Gauss map. In this note we give a self-contained proof for \(?(x)\). Several aspects of our argument are substantially different, making a very short proof possible in the \(?(x)\) case. Moreover, we obtain an explicit polynomial decay estimate, with exponent at least \(0.0472\). We then prove the existence of the correlation dimension of the Minkowski measure, thereby improving the previous upper bound for the optimal Fourier decay exponent from \(0.4373_{+}\) to \(0.4223_{+}\).

\end{abstract}

\maketitle
%\begin{center}
%Mathematical Analysis/Number Theory
%\end{center}

\section{The coefficients and the measure}
 Recall that the Minkowski question mark function is defined, using a representation of $x\in[0,1]$ by a continued fraction, by
\[
	?([0,a_1,a_2,\ldots])
	=2\sum_{j=1}^{\infty}(-1)^{j+1}2^{-(a_1+\cdots+a_j)},
	\quad a_j\in\mathbb N.
\]
Following the notation of \cite{AlkCR}, put
\[
\m(t)=\int_0^1e^{xt}\d ?(x),\qquad
d_n=\m(2\pi i n)=\int_0^1\cos(2\pi nx)\d ?(x).
\]
In was shown in \cite{AlkFS,AlkCR} that $\m(it)$ satisfies an integral equation, and coefficients $d_{n}$ satisfy discrete equations. The present note, however, is independent of those results. In 1943, Salem~\cite{Salem} asked whether $d_n\to0$. This was answered in \cite{JS}  as a consequence of their general results on Gibbs measures for the Gauss map.
\begin{thm}[Jordan--Sahlsten] As $|t|\rightarrow\infty$, $|\m(it)|\ll |t|^{-\eta}$ for a certain $\eta>0$.
\end{thm}	
The first aim of this note is, in case of the measure $d?(x)$, to simplify the proof and make the decay exponent explicit. We follow the same branch-pair $L^2$ strategy of Kaufman~\cite{Kaufman} and
Queff\'elec--Ramar\'e~\cite{QR}, but specialize it to the Minkowski measure, perform one major and  two minor simplifications. 

\begin{thm}\label{main} As $|t|\rightarrow\infty$, $|\m(it)|\ll |t|^{-1/35}$.
\end{thm}

Write \(\mu=\d ?\). If \(A_1,A_2,\ldots\) are independent random variables with \(\Pp(A_j=a)=2^{-a}\), then the random continued fraction $[0,A_1,A_2,\ldots]$ has distribution function \(?(x)\). Since \(?(x)\) is continuous, the associated measure has no atoms. 
Now, with
\[
h_a(x)=\frac1{a+x},\quad {\mathbf a}=(a_1,\ldots a_k),\quad h_{\mathbf a}=h_{a_1}\circ\cdots\circ h_{a_k},
\quad I_{\mathbf a}=h_{\mathbf a}([0,1]),\quad
w_{\mathbf a}=2^{-(a_1+\cdots+a_k)},
\]
we have, for every bounded Borel function $g$,
\begin{equation}\label{cylinder}
	\int_{I_{\mathbf a}}g\d\mu
	=w_{\mathbf a}\int_0^1g(h_{\mathbf a}(x))\d\mu(x).
\end{equation}
Endpoints have zero mass throughout. These are standard facts about $?(x)$. We use the classical sharp H\"older estimate for $?(x)$ determined by Salem~\cite{Salem}: for every interval $J\subseteq[0,1]$,
\begin{equation}\label{holder}
	\mu(J)\ll |J|^\alpha\le |J|^{2/3},\quad
	\alpha=\frac{\log2}{2\log \varphi}=0.7202_{+}, \text{where }\varphi=\frac{1+\sqrt5}{2}.
\end{equation}

\section{Stopping at a denominator scale}

For convergents $p_j/q_j=[0,a_1,\ldots,a_j]$, use
$p_{-1}=1,p_0=0,q_{-1}=0,q_0=1$ and
$q_j=a_jq_{j-1}+q_{j-2}$, with the same recurrence for $p_j$.
Here are the standard branch formulas \cite{Kh}:
\begin{equation}\label{branch}
	h_{\mathbf a}(x)=\frac{p_k+p_{k-1}x}{q_k+q_{k-1}x},\qquad
	h_{\mathbf a}'(x)=\frac{(-1)^k}{(q_k+q_{k-1}x)^2},\qquad
	|I_{\mathbf a}|=\frac1{q_k(q_k+q_{k-1})}.
\end{equation}

The first major simplification relative to Jordan-Sahlsten is the use of a denominator stopping time, while in \cite{JS} authors use large deviations to select regular intervals at a fixed generation and refined continuant estimates.\\

So, fix $Q\ge2$. Stop each digit sequence at the first even index $k$ with
$q_k\ge Q$, and denote the family of stopping words $\mathbf{a}$ by $\W_Q$.
These intervals partition the irrationals; even length makes every branch
increasing. For $\mathbf a\in\W_Q$, write
$q_{\mathbf a}=q_k$ and $r_{\mathbf a}=q_{k-1}$. Then
\begin{equation}\label{weights}
	\sum_{\W_Q}w_{\mathbf a}=1,\qquad
	q_{\mathbf a}\ge Q,\qquad
	0<r_{\mathbf a}<q_{\mathbf a},\qquad
	\max_{\W_Q}w_{\mathbf a}\ll Q^{-4/3},
\end{equation}
the last estimate following from \eqref{holder} and \eqref{branch}.
The following result controls the overshoot of the denominators $q_{\mathbf a}$ beyond $Q$.
\begin{lem}\label{moment}
	Uniformly for $Q\ge2$,
	\[
		\sum_{\mathbf a\in\W_Q}w_{\mathbf a}q_{\mathbf a}^{3/2}\ll Q^{3/2}.
	\]
\end{lem}
\begin{proof}
	For a random digit sequence put $\tau=\min\{j\ge1:q_{2j}\ge Q\}$ and
	$Z_j=(a_{2j-1}+1)(a_{2j}+1)$. The recurrences for convergents give
	\[
	2q_{2j-2}\le q_{2j}\le Z_jq_{2j-2}.
	\]
	Thus $\tau\le\lceil\log_2Q\rceil$. The event $\{\tau\ge j\}$
	and $q_{2j-2}$ depend only on the preceding digits, hence are independent
	of $Z_j$. Writing
	$C_*=(\sum_{a\ge1}2^{-a}(a+1)^{3/2})^2<\infty$, we obtain
	\[
	\E q_{2\tau}^{3/2}
	\le\sum_{j\ge1}\E\bigl[\ind_{\{\tau\ge j\}}Z_j^{3/2}q_{2j-2}^{3/2}\bigr]
	=C_*\E\sum_{j=1}^{\tau}q_{2j-2}^{3/2}
	\le\frac{C_*Q^{3/2}}{1-2^{-3/2}}.
	\]
	The last inequality holds on each digit sequence: the denominators before
	stopping are below $Q$ and increase by at least a factor $2$ at every even
	step. The expectation on the left is exactly the sum in the lemma.
\end{proof}

The pairs $(q_{\mathbf a},r_{\mathbf a})$ are primitive and distinct. Indeed, even length gives $q_{\mathbf a}p_{k-1}-r_{\mathbf a}p_k=1$.
It determines $p_k$ uniquely modulo $q_{\mathbf a}$, and $0<p_k<q_{\mathbf a}$;
then $p_{k-1}$ is determined as well. Equal pairs would therefore give equal
branches, contrary to disjointness of stopping-interval interiors.

\section{Two elementary analytic estimates}

For distinct stopping words put
$D(\mathbf a,\mathbf b)=\max\{|q_{\mathbf a}-q_{\mathbf b}|,
|r_{\mathbf a}-r_{\mathbf b}|\}\ge1$.

\begin{lem}\label{osc}
	For $t>0$ and $M=\max(q_{\mathbf a},q_{\mathbf b})$,
	\[
		\left|\int_0^1e^{it(h_{\mathbf a}(x)-h_{\mathbf b}(x))}\d x\right|
		\ll\left(\frac{M^3}{tD(\mathbf a,\mathbf b)}\right)^{1/2}.
	\]
\end{lem}
\begin{proof}
Set $U=q_{\mathbf a}+r_{\mathbf a}x$, $V=q_{\mathbf b}+r_{\mathbf b}x$,
$\psi=h_{\mathbf a}-h_{\mathbf b}$, and $D=D(\mathbf a,\mathbf b)$. Then
\[
\psi'=(V-U)\frac{U+V}{U^2V^2},\qquad
|\psi'|\ge\frac{|V-U|}{4M^3}.
\]
Moreover, $\psi'$ is monotone on at most two intervals: $\psi''=0$ is
equivalent to $U/V=(r_{\mathbf a}/r_{\mathbf b})^{1/3}$, and $U/V$ is
strictly monotone for distinct primitive pairs.

For $0<\varepsilon\le1/2$, the interval
$\{|V-U|<\varepsilon D\}$ has length at most $4\varepsilon$.
Indeed, either the slope of $V-U$ has magnitude at least $D/2$, or its
constant term has magnitude $D$ and this interval is empty.
On the complement, $|\psi'|\ge\varepsilon D/(4M^3)$.
The first-derivative van der Corput estimate
\cite[Theorem~14.2]{Mattila}, applied on the bounded number of monotonicity
intervals, therefore gives
\[
\left|\int_0^1e^{it\psi(x)}\d x\right|
\ll\varepsilon+\frac{M^3}{t\varepsilon D}.
\]
Choose $\varepsilon=(M^3/(tD))^{1/2}$ when this is at most $1/2$;
otherwise use the trivial bound $1$.
\end{proof}

\begin{lem}\label{transfer}
	If $g\in C^1([0,1])$, $\norm{g}_\infty\le1$, $\norm{g'}_\infty\le L$
	with $L\ge1$, and $E_g=\int_0^1|g|^2\d x$, then
	\[
		\int_0^1|g|\d\mu\ll E_g^{3/7}L^{1/7}.
	\]
\end{lem}
\begin{proof}
Partition $[0,1]$ into equal intervals $J$ of length $\ell$.
Averaging $|g(x)|\le |g(y)|+L\ell$ over $y\in J$, then applying
Cauchy--Schwarz first on $J$ and then to the sum weighted by $\mu(J)$, gives
\[
\int|g|\d\mu
\le\sum_J\mu(J)\left(\ell^{-1}\int_J|g|^2\d x\right)^{1/2}+L\ell
\ll E_g^{1/2}\ell^{-1/6}+L\ell,
\]
by \eqref{holder}. For $E_g>0$, take
$\ell\asymp(E_g/L^2)^{3/7}\le1$. If $E_g=0$, then $g=0$.
\end{proof}

\section{Proof of the theorem}

For $t>0$, average over all stopping words:
\[
f(x)=\sum_{\mathbf a\in\W_Q}w_{\mathbf a}e^{it h_{\mathbf a}(x)}.
\]
The series and its derivative converge uniformly by \eqref{weights}
and \eqref{branch}. Equation~\eqref{cylinder} therefore gives
\begin{equation}\label{reduction}
	|\m(it)|\le\int_0^1|f|\d\mu,\qquad
	\norm{f}_\infty\le1,\qquad \norm{f'}_\infty\le tQ^{-2}.
\end{equation}
For a fixed denominator pair, at most $8j$ integer pairs lie at
maximum-coordinate distance $j$. Splitting at $j=Q^{2/3}$ and using
\eqref{weights}, we obtain, uniformly in $\mathbf a$,
\begin{equation}\label{shell}
\sum_{\mathbf b\ne\mathbf a}w_{\mathbf b}D(\mathbf a,\mathbf b)^{-1/2}
\ll Q^{-4/3}\sum_{1\le j\le Q^{2/3}}j^{1/2}+Q^{-1/3}
\ll Q^{-1/3}.
\end{equation}
Expand $|f|^2$; the double series may be integrated termwise since
$\sum_{\mathbf a,\mathbf b}w_{\mathbf a}w_{\mathbf b}=1$.
The diagonal is at most $\max w_{\mathbf a}\ll Q^{-4/3}$.
For the off-diagonal use Lemma~\ref{osc},
$M^{3/2}\le q_{\mathbf a}^{3/2}+q_{\mathbf b}^{3/2}$, symmetry,
\eqref{shell}, and Lemma~\ref{moment}. It follows that
\[
\begin{aligned}
	E:=\int_0^1|f|^2\d x
	&\ll Q^{-4/3}
	+t^{-1/2}\sum_{\mathbf a}w_{\mathbf a}q_{\mathbf a}^{3/2}
	\sum_{\mathbf b\ne\mathbf a}
	w_{\mathbf b}D(\mathbf a,\mathbf b)^{-1/2}\\
	&\ll Q^{-4/3}+t^{-1/2}Q^{7/6}.
\end{aligned}
\]
For $t\ge32$, choose $Q=t^{1/5}$. Then
$E\ll Q^{-4/3}$ and $L:=tQ^{-2}=Q^3$.
Lemma~\ref{transfer} and \eqref{reduction} give
\[
|\m(it)|\ll E^{3/7}L^{1/7}
\ll Q^{-4/7}Q^{3/7}=Q^{-1/7}=t^{-1/35}.
\]
Conjugation handles negative $t$, and $|\m(it)|\le1$ handles the remaining
bounded range. This proves Theorem~\ref{main}.\hfill$\square$\\

Jordan--Sahlsten also organize branch pairs by continuant differences
and reduce integration against the Gibbs measure to Lebesgue $L^2$
estimates. Here the denominator stopping rule and exact Bernoulli
weights permit elementary lattice counting, while the global
H\"older bound gives the direct averaging estimate of
Lemma~\ref{transfer}.\\

The choice $2/3$ in \eqref{holder} keeps the proof simple.
Retaining $\alpha$ and choosing $Q=t^{1/[3(1+\alpha)]}$ gives 
\[
\eta=\frac{3\alpha^2-1}{3(1+\alpha)(3-\alpha)}
=0.0472673444\ldots .
\]
 Using more complicated technique, we could push thist a bit further. However, it seems to be extremely difficult, in the framework of the current approach, to get above $0.1$. Thus, as for the lower bound, we aimed for simplicity and not the constant. Yet, slightly improving the known upper bound, we can pose a cautious conjecture about the actual decay speed.\\

% The following final section is independent of the preceding pointwise-decay proof.

\section{Correlation dimension}

As one of the methods to attack Salem's problem, the following quadratic sums were introduced in \cite{AlkFS}:
\[
\Sigma_N=\sum_{j=0}^{N-1}\Bigg{(}?\!\left(\frac{j+1}{N}\right)-?\!\left(\frac jN\right)\Bigg{)}^2.
\]
The standard comparisons with correlation integrals and Fourier mean
squares are
\begin{equation}\label{correlationcomparison}
	\Sigma_N\asymp(\mu\otimes\mu)\{|x-y|\le N^{-1}\}
	\asymp\frac1N\int_{-N}^{N}|\m(it)|^2\d t;
\end{equation}
see \cite{GH}. For the second comparison one may use the positive
triangular Fourier cutoff; the first comparison also gives
$C(Ar)\ll A\,C(r)$ for $C(r)=(\mu\otimes\mu)\{|x-y|\le r\}$, and a
dyadic decomposition yields the reverse estimate. We therefore need only
determine $\Sigma_N$ from the Minkowski interval weights.

\begin{prop}
	The correlation dimension
	$D_2=-\lim_{N\to\infty}\log\Sigma_N/\log N$ exists and is the unique zero
	in $(0,1)$ of $P(s)=\log r(\mathcal K_s)$, where $r$ is the spectral
	radius on $C([0,1])$ and, for $s\ge0$,
	\begin{equation}\label{pressure}
		(\mathcal K_sg)(x)=\sum_{a\ge1}4^{-a}(a+x)^{2s}
		g\left(\frac{1}{a+x}\right).
	\end{equation}
	Moreover, $\Sigma_N\asymp N^{-D_2}$ and
	$\int_{-T}^{T}|\m(it)|^2\d t\asymp T^{1-D_2}$ for $T\ge1$.
\end{prop}

\begin{proof}
	\emph{Pressure.}
	For arbitrary words let
	$Z_n(s)=\sum_{|\mathbf a|=n}w_{\mathbf a}^2q_{\mathbf a}^{2s}$, with
	$Z_0=1$. The continuant bounds
	$q_{\mathbf a}q_{\mathbf b}\le q_{\mathbf a\mathbf b}
	\le2q_{\mathbf a}q_{\mathbf b}$ give
	\[
	Z_nZ_m\le Z_{n+m}\le4^sZ_nZ_m,
	\qquad
	Z_n\le\mathcal K_s^n1\le4^sZ_n.
	\]
	Gelfand's spectral-radius formula therefore gives
	\begin{equation}\label{partitionpressure}
		4^{-s}e^{nP(s)}\le Z_n(s)\le e^{nP(s)}.
	\end{equation}
	The bound $q_{\mathbf a}\le\prod_j(a_j+1)$ makes $P$ finite, and
	\eqref{partitionpressure} shows that
	$P(s)=\lim_{n\to\infty}n^{-1}\log Z_n(s)$; hence $P$ is convex.
	The Fibonacci bound gives
	$P(t)-P(s)\ge2(t-s)\log\varphi$ for $t>s$.
	Moreover, $P(0)=-\log3$, and the bounds with $n=1$ show that
	$P(s)\to P(0)$ as $s\downarrow0$. Finally,
	$Z_2(1)=\sum_{a,b\ge1}4^{-a-b}(ab+1)^2=769/729>1$, so $P(1)>0$.
	Thus $P$ has a unique zero $\delta\in(0,1)$.
	
	\emph{Stopping sums.}
	Stop at the first denominator at least $Q\ge2$, without a parity
	restriction:
	$\mathcal V_Q=\{\mathbf a:q_{\mathbf a^-}<Q\le q_{\mathbf a}\}$,
	where $\mathbf a^-$ deletes the last digit.
	Extend all stopping words to a common depth, which exists by the
	Fibonacci bound. The continuant bounds and \eqref{partitionpressure}
	for the tails give
	\begin{equation}\label{criticalsum}
		\sum_{\mathcal V_Q}w_{\mathbf a}^2q_{\mathbf a}^{2\delta}\asymp1.
	\end{equation}
	Group these words by their parent $\mathbf b$. If $a_0$ is its first
	crossing digit, then for $k\ge0$,
	\[
	Q\le q_{\mathbf b(a_0+k)}<(k+2)Q,\qquad
	w_{\mathbf b(a_0+k)}^2=4^{-k}w_{\mathbf b a_0}^2.
	\]
	Summing the geometric tails, first with the power $2\delta$ and then
	with power $2$, turns \eqref{criticalsum} into
	\begin{equation}\label{stoppedmoments}
		\sum_{\mathcal V_Q}w_{\mathbf a}^2\asymp Q^{-2\delta},
		\qquad
		\sum_{\mathcal V_Q}w_{\mathbf a}^2q_{\mathbf a}^2
		\ll Q^{2-2\delta}.
	\end{equation}
	
	\emph{Comparison with the grid.}
	Take $Q=\sqrt N$. Every stopped interval $I$ has length at most
	$N^{-1}$ and meets at most two grid cells. Conversely, the stopped
	intervals meeting a given cell $J$ have total length at most $3/N$.
	Hence
	\begin{equation}\label{gridcomparison}
		\frac12\sum_{\mathcal V_Q}w_{\mathbf a}^2
		\le\Sigma_N\le\frac3N\sum_{\mathcal V_Q}
		\frac{w_{\mathbf a}^2}{|I_{\mathbf a}|}
		\le\frac6N\sum_{\mathcal V_Q}w_{\mathbf a}^2q_{\mathbf a}^2.
	\end{equation}
	For the upper bound, apply Cauchy--Schwarz to
	$\mu(J)=\sum_I\mu(I\cap J)$ with weights $|I|$, then sum over $J$;
	use $\sum_J\mu(I\cap J)^2\le\mu(I)^2$.
	Equations \eqref{stoppedmoments}--\eqref{gridcomparison} give
	$\Sigma_N\asymp N^{-\delta}$, and \eqref{correlationcomparison}
	proves the remaining assertions.
\end{proof}

The Hausdorff dimension of the Minkowski measure was calculated to
$36$ decimal digits in \cite[Appendix~A.3]{AlkMC}:
\[
\dim_H\mu=
\frac{\log2}{2\int_0^1\log(1+x)\d\mu(x)}
=0.874716305108211142215152904219159757\ldots .
\]
Let
$\eta_*=\sup\{\eta\ge0:|\m(it)|=O(|t|^{-\eta})\}$ be the optimal
Fourier decay exponent. By the classical Fourier--energy identity and
the standard inequalities between Fourier, correlation and Hausdorff
dimensions; see, for example, \cite[Secs.~2.5, 3.5--3.6]{Mattila}
and \cite{GH},
\[
\min\{1/2,\eta_*\}\le\frac{D_2}{2}\le\frac{\dim_H\mu}{2}.
\]
Since $D_2<1$, this gives $\eta_*\le D_2/2$.

Numerical evaluation of \eqref{pressure} gives
$D_2/2=0.4223047_{+}$, compared
with $\dim_H\mu/2=0.4373582_{+}$. Thus the correlation dimension
gives a numerically sharper upper bound for the optimal Fourier decay
exponent. The mean-square law suggests the cautious conjecture
$\eta_*=D_2/2$.

\noindent Vilnius University, Department of Mathematics and Informatics, Institute of Computer Science, Naugarduko 24, LT-03225 Vilnius, Lithuania.
{\tt giedrius.alkauskas@mif.vu.lt}

\end{document}